\documentclass[11pt,english,reqno]{article}
\usepackage{lmodern}
\usepackage{lmodern}

\usepackage[T1]{fontenc}
\usepackage[latin9]{inputenc}
\usepackage{float}
\usepackage{enumitem}
\usepackage{tipa}
\usepackage{tipx}
\usepackage{amsmath}
\usepackage{amsthm}
\usepackage{amssymb}
\usepackage{graphicx}
\usepackage{color}

\makeatletter
\usepackage{fullpage}
\usepackage{relsize}

\newtheoremstyle{myremark} % name
    {\topsep}                    % Space above
    {\topsep}                    % Space below
    {\rm}                        % Body font
    {}                           % Indent amount
    {\bf}                        % Theorem head font
    {.}                          % Punctuation after theorem head
    {.5em}                       % Space after theorem head
    {}  % Theorem head spec (can be left empty, meaning normal)
    
\theoremstyle{myremark}

\usepackage{listings}
\usepackage{amsmath}
\usepackage{amsthm}
\usepackage{tikz}
\usepackage{caption}
\usepackage{array}
\usepackage{mdwmath}
\usepackage{multirow}
\usepackage{mdwtab}
\usepackage{eqparbox}
\usepackage{amsfonts}
\usepackage{tikz}
\usepackage{multirow,bigstrut,threeparttable}
\usepackage{amsthm}
\usepackage{array}
\usepackage{bbm}
\usepackage{subfigure}
\usepackage{epstopdf}
\usepackage{mdwmath}
\usepackage{mdwtab}
\usepackage{eqparbox}
\usepackage{tikz}
\usepackage{latexsym}
\usepackage{cite}
\usepackage{amssymb}
\usepackage{bm}
\usepackage{amssymb}
\usepackage{graphicx}
\usepackage{mathrsfs}
\usepackage{epsfig}
\usepackage{psfrag}
\usepackage{setspace}
\usepackage{hyperref}
\usepackage{algorithm}
\usepackage{algpseudocode}
\usepackage{stfloats}

\newtheorem{theorem}{Theorem}

\newtheorem{lemma}{Lemma}

\newcommand{\norm}[1]{\left\lVert#1\right\rVert}

\DeclareMathOperator*{\argmax}{arg\,max}
\DeclareMathOperator*{\esssup}{ess\,sup}

\begin{document}

\title{Generalizations of Maximal Inequalities to Arbitrary Selection Rules}


\author{Jiantao Jiao, Yanjun Han, and Tsachy Weissman
\thanks{Jiantao Jiao, Yanjun Han, and Tsachy Weissman are with the Department of Electrical Engineering, Stanford University, CA, USA. Email: \{jiantao,yjhan, tsachy\}@stanford.edu.}
}

\date{\today}

% \vspace{-10pt}

%
\maketitle

\begin{abstract}
We present a generalization of the maximal inequalities that upper bound the expectation of the maximum of $n$ jointly distributed random variables. We control the expectation of a randomly selected random variable from $n$ jointly distributed random variables, and present bounds that are at least as tight as the classical maximal inequalities, and much tighter when the distribution of selection index is near deterministic. A new family of information theoretic measures were introduced in the process, which may be of independent interest.
\end{abstract}
%
%\begin{IEEEkeywords}
%XXX
%\end{IEEEkeywords}

\section{Introduction}
Throughout this paper, we consider $n$ random variables $Z_i, 1\leq i\leq n$ such that $\mathbb{E}[Z_i] = 0$, where $n$ is a finite positive integer. The zero mean condition can be satisfied via the operation $Z_i' = Z_i - \mathbb{E}[Z_i]$ upon assuming that all $Z_i$'s are integrable. The following two maximal inequalities are well known in the literature and serve as the motivational results for this work.
\begin{lemma}\label{lemma.mgfhard}
Let $\psi \geq 0$ be a convex function defined on the interval $[0,b)$ where $0<b\leq \infty$. Assume that $\psi(0) = 0$. Set, for every $t\geq 0$,
\begin{align}
\psi^*(t) & = \sup_{\lambda \in (0,b)} (\lambda t - \psi(\lambda)).
\end{align}
Suppose that $\ln \mathbb{E}[e^{\lambda Z_i}] \leq \psi(\lambda)$ for all $\lambda \in [0,b), 1\leq i\leq n$. Then,
\begin{align}
\mathbb{E}[\max_i Z_i] & \leq \psi^{*-1}(\ln n),
\end{align}
where $\psi^{*-1}(y)$ is defined as
\begin{align}
\psi^{*-1}(y) = \inf \{t\geq 0: \psi^*(t)>y\}.
\end{align}
\end{lemma}

To introduce the second inequality, we say a function $\psi$ is an Orlicz function if $\psi: [0,\infty) \mapsto [0,\infty]$ is a convex function vanishing at zero and is also not identically $0$ or $\infty$ on $(0,\infty)$. We define the Luxemburg $\psi$ norm of a random variable $X$ as
\begin{align}
\norm{X}_\psi = \inf \left \{ \sigma>0: \mathbb{E}\left[ \psi \left( \frac{|X|}{\sigma} \right) \right] \leq 1 \right \}.
\end{align}
\begin{lemma}\label{lemma.orliczhard}
\cite{pollardonline2016} Let $\psi$ be an Orlicz function. Suppose $\norm{ Z_i}_\psi \leq \sigma, 1\leq i\leq n$. Then,
\begin{align}
\mathbb{E}[\max_i Z_i] & \leq  \sigma \cdot \psi^{-1}(n),
\end{align}
where $\psi^{-1}(y)$ is defined as $\psi^{-1}(y) = \inf \{t\geq 0: \psi(t)> y\}$.
\end{lemma}

This paper generalizes Lemma~\ref{lemma.mgfhard} and~\ref{lemma.orliczhard} to arbitrary selection rules. Concretely, suppose $T\in \{1,2,\ldots,n\}$ is a random variable jointly distributed with $Z_1,Z_2,\ldots,Z_n$. We would like to upper bound $\mathbb{E}[Z_T]$, which subsumes the maximal inequality $T = \argmax_i Z_i$ as a special case. Naturally, since
\begin{align}
\mathbb{E}[Z_T] & \leq \mathbb{E}[\max_i Z_i],
\end{align}
we would like to obtain bounds that are at least as strong as Lemma~\ref{lemma.mgfhard} and~\ref{lemma.orliczhard}, but dependent on the joint distribution of $T, Z_1,Z_2,\ldots,Z_n$. In particular, the upper bound should be zero if $T$ is deterministic since we have already assumed that $\mathbb{E}[Z_i] = 0$ for all $1\leq i\leq n$.

A generalization of Lemma~\ref{lemma.mgfhard} was achieved in~\cite{jiao--han--weissman17dependence} using the Donsker--Varadhan representation of the relative entropy, which is a generalization of the sub-Gaussian case in~\cite{Russo--Zou2015much}. Denote the entropy of a discrete random variable $T$ as
\begin{align}
H(T) & = \sum_{t} P_T(t) \ln \frac{1}{P_T(t)},
\end{align}
and the mutual information $I(X;Y)$ between $X$ and $Y$ as
\begin{align}
I(X;Y) = \begin{cases} \int \ln \frac{dP_{XY}}{d(P_XP_Y)} dP_{XY} & \text{if } P_{XY} \ll P_XP_Y \\ \infty & \text{otherwise} \end{cases}.
\end{align}
The following was shown in~\cite{jiao--han--weissman17dependence}.
\begin{lemma}\label{lemma.mgfsoft}
Let $\psi \geq 0$ be a convex function defined on the interval $[0,b)$ where $0<b\leq \infty$. Assume that $\psi(0) = 0$. Set, for every $t\geq 0$,
\begin{align}
\psi^*(t) & = \sup_{\lambda \in (0,b)} (\lambda t - \psi(\lambda)).
\end{align}
Suppose that $\ln \mathbb{E}[e^{\lambda Z_i}] \leq \psi(\lambda)$ for all $\lambda \in [0,b), 1\leq i\leq n$, and $\mathbb{E}[Z_i] = 0, 1\leq i\leq n$. Then,
\begin{align}
\mathbb{E}[Z_T] & \leq \psi^{*-1}(I(T;\mathbf{Z})) \\
& \leq \psi^{*-1}(H(T)) \\
\end{align}
where $\psi^{*-1}(y)$ is defined as
\begin{align}
\psi^{*-1}(y) = \inf \{t\geq 0: \psi^*(t)>y\}.
\end{align}
and $\mathbf{Z} = (Z_1,Z_2,\ldots,Z_n)$.
\end{lemma}
Lemma~\ref{lemma.mgfsoft} is clearly stronger than Lemma~\ref{lemma.mgfhard} since $I(T;\mathbf{Z})  \leq H(T) \leq \ln n$. It is also interesting to observe that the soft bound is maximized when $T$ follows a uniform distribution, and it is zero when $T$ is deterministic.

Similar attempts were made to generalize Lemma~\ref{lemma.orliczhard} in~\cite{jiao--han--weissman17dependence}. However, it was not satisfactory since that even in the case of $\psi(x) = x^p, p\geq 1, x\geq 0$, the generalization bound obtained in~\cite{jiao--han--weissman17dependence} may be infinity when $1\leq p<2$, while Lemma~\ref{lemma.orliczhard} shows that it is universally bounded by $\sigma \cdot n^{1/p}$ for every $p\geq 1$.

Our main contribution in this paper is the generalization of Lemma~\ref{lemma.orliczhard} to arbitrary selection rules. Our generalization satisfies the following properties:
\begin{enumerate}
\item It is at least as strong as Lemma~\ref{lemma.orliczhard}: in other words, it can be shown that the worst case joint distribution of $T$ and $\mathbf{Z}$ would not incur an upper bound larger than $\sigma \cdot \psi^{-1}(n)$, which is the upper bound in Lemma~\ref{lemma.orliczhard}.
\item It admits a closed form expression for the $p$-norm case, i.e., the case where $\psi(x) = x^p, p\geq 1,x\geq 0$. In other words, it defines another information theoretic measure paralleling the Shannon entropy $H(T)$ in Lemma~\ref{lemma.mgfsoft}. Concretely, for any $1\leq q\leq \infty$, we introduce functional $H(T;q)$ as
\begin{align}\label{eqn.hqdef}
H(T;q) & = \begin{cases} \frac{1}{2}\mathbbm{1}(H(T)\neq 0) & q = \infty \\
 \left( \sum_{t} \left( P_T(t)^{1/(1-q)} + (1-P_T(t))^{1/(1-q)} \right)^{1-q}  \right)^{1/q} & 1<q<\infty \\ \sum_t \min\{P_T(t), 1-P_T(t)\} & q = 1 \end{cases},
\end{align}
where $\mathbbm{1}(A) = \begin{cases} 1 & A\text{ is true }\\ 0 & \text{otherwise} \end{cases}$, and $H(T)$ is the Shannon entropy functional. The $H(T;q)$ functional satisfies the following properties:
\begin{enumerate}
\item $0\leq H(T;q) \leq 1$;
\item $H(T;q) = 0 \Leftrightarrow T$ is deterministic.
\end{enumerate}
\end{enumerate}

The rest of the paper is organized as follows. We present and discuss our main results in Section~\ref{sec.mainresults}. Auxiliary lemmas and their proofs are provided in Section~\ref{sec.auxiliary}, and the proofs of Lemma~\ref{lemma.mgfhard} and~\ref{lemma.orliczhard} are provided in Section~\ref{sec.classicalproof} for completeness.

\section{Preliminaries}
The $\beta$-norm of a random variable $X$ for $\beta\geq 1$ is defined as
\begin{align}
\norm{X}_\beta = \begin{cases}  (\mathbb{E} |X|^\beta)^{1/\beta} & 1\leq \beta <\infty \\
\esssup |X| & \beta = \infty \end{cases},
\end{align}
where the essential supremum is defined as
\begin{align}
\esssup X = \inf\{M: \mathbb{P}(X>M) = 0\}.
\end{align}

The Fenchel--Young inequality states that for any function $f$ and its convex conjugate $f^*$, we have
\begin{align}
f(x) + f^*(y) \geq \langle x, y \rangle, \textrm{ for all }x\in X, y\in X^*,
\end{align}
which follows from the definition of convex conjugate $f^*(y) = \sup_{x\in X} \{ \langle x, y\rangle - f(x) \}$. It follows from the Fenchel--Moreau theorem that $f = f^{**}$ if and only if $f$ is convex and lower semi-continuous. Note that any convex function $f: [0,\infty) \mapsto [0,\infty]$ that satisfies $f(0) = 0$ is lower semi-continuous.

We define the Ameniya norm of a random variable $X$ as
\begin{align}
\norm{X}_\psi^A = \inf \left \{ \frac{1 + \mathbb{E} \psi(|tX|)}{t}: t>0 \right \}.
\end{align}

\section{Main results}\label{sec.mainresults}

We present our main result below.
\begin{theorem}\label{thm.main}
Let $\psi$ be an Orlicz function. Suppose $\norm{Z_i}_\psi \leq \sigma, \mathbb{E}[Z_i]  = 0, 1\leq i\leq n$. Then,
\begin{align}
|\mathbb{E}[Z_T]| & \leq \sigma \cdot \sum_{i = 1}^n \inf_{a_i} \norm{P_{T|\mathbf{Z}}(i|\mathbf{Z}) - a_i}_{\psi^*}^A \\
& \leq \sigma \cdot \inf_{t>0} \frac{1}{t} \left( n + \sum_{i = 1}^n P_T(i)\psi^*(t|1-a_i|) + (1-P_T(i)) \psi^*(t|a_i|) \right).
\end{align}
Furthermore, if $\norm{Z_i}_p \leq \sigma, p\geq 1, \frac{1}{p} + \frac{1}{q} = 1$, then
\begin{align}
|\mathbb{E}[Z_T]| & \leq  \sigma \cdot n^{1/p} \left(  \sum_{i = 1}^n  \inf_{a_i\in \mathbb{R}}\mathbb{E}|P_{T|\mathbf{Z}}(i|\mathbf{Z}) - a_i|^q \right)^{1/q} \\
& \leq \sigma \cdot n^{1/p} H(T;q) ,
\end{align}
where $\mathbf{Z} = (Z_1,Z_2,\ldots,Z_n)$, and $H(T;q)$ is defined in~(\ref{eqn.hqdef}).
\end{theorem}

\begin{proof}
For any $t>0$, we have the following chain of inequalities:
\begin{align}
\mathbb{E}\left[ \frac{Z_T}{\sigma} \right] & = \sum_{i = 1}^n P_T(i) \mathbb{E}\left[ \frac{Z_i}{\sigma} \Bigg | T = i\right]\\
& = \sum_{i = 1}^n P_T(i) \int \frac{P_{Z_i|T = i}(dx)}{P_{Z_i}(dx)} \frac{x}{\sigma} P_{Z_i}(dx) \\
& = \sum_{i = 1}^n P_T(i) \int \left( \frac{P_{Z_i|T = i}(dx)}{P_{Z_i}(dx)} -b_i \right) \frac{x}{\sigma} P_{Z_i}(dx) \\
& = \sum_{i = 1}^n  \int \left( \frac{P_{Z_i,T = i}(dx)}{P_{Z_i}(dx)} -b_i P_T(i) \right) \frac{x}{\sigma} P_{Z_i}(dx) \\
& \leq \sum_{i = 1}^n \int \left | P_{T|Z_i}(i|x) - a_i \right | \frac{|x|}{\sigma} P_{Z_i}(dx),
\end{align}
where $a_i = b_i P_T(i)$, and the vectors $(a_1,a_2,\ldots,a_n)^T$ and $(b_1,b_2,\ldots,b_n)^T$ are deterministic vectors in $\mathbb{R}^n$. The derivations above hold for any arbitrary vector $(a_1,a_2,\ldots,a_n)^T \in \mathbb{R}^n$.

Applying the generalized Holder's inequality, we obtain that
\begin{align}
\mathbb{E}\left[ \frac{Z_T}{\sigma} \right] & \leq \sum_{i = 1}^n \inf_{a_i\in \mathbb{R}} \norm{P_{T|Z_i}(i|Z_i)-a_i}_{\psi^*}^A.
\end{align}

We further upper bound each term in the summation as follows. For each $i, 1\leq i\leq n$,
\begin{align}
\norm{ P_{T|Z_i}(i|Z_i) - a_i }_{\psi^*}^A & = \inf_{t>0} \frac{1 + \mathbb{E}[\psi^*(t|P_{T|Z_i}(i|Z_i) - a_i|)]}{t} \\
& \leq \inf_{t>0} \frac{1 + \mathbb{E}[\psi^*(t|P_{T|\mathbf{Z}}(i|\mathbf{Z}) - a_i|)]}{t} \\
& = \norm{ P_{T|\mathbf{Z}}(i|\mathbf{Z}) - a_i }_{\psi^*}^A.
\end{align}
Here in the second step we have used the fact that $\psi^*(t|x|)$ is a convex function of $x$, and the fact that $\mathbb{E}[P_{T|\mathbf{Z}}(i|\mathbf{Z}) | Z_i] = P_{T|Z_i}(i|Z_i)$.

Hence, we have proved that
\begin{align}\label{eqn.firstbound}
\mathbb{E}[Z_T] & \leq \sigma \cdot \sum_{i = 1}^n \inf_{a_i} \norm{ P_{T|\mathbf{Z}}(i|\mathbf{Z}) - a_i }_{\psi^*}^A.
\end{align}
It is clear that the inequality above also holds for $-\mathbb{E}\left[ Z_T \right]$. Hence, one has
\begin{align}
|\mathbb{E}[Z_T]| & \leq \sigma \cdot \sum_{i = 1}^n \inf_{a_i} \norm{ P_{T|\mathbf{Z}}(i|\mathbf{Z}) - a_i }_{\psi^*}^A.
\end{align}

We now further upper bound the RHS of~(\ref{eqn.firstbound}) to obtain a bound that only depends on the marginal distribution of $T$ but not the joint distribution of $T$ and $\mathbf{Z}$. For any $t>0,a_1,a_2,\ldots,a_n\in \mathbb{R}$, we have
\begin{align}
\mathbb{E}[Z_T] & \leq \sigma \cdot \frac{1}{t} \left( n + \mathbb{E} \left[ \sum_{i = 1}^n \psi^*(t|P_{T|\mathbf{Z}}(i|\mathbf{Z})-a_i|) \right] \right).
\end{align}
Since $\psi^*(t|x-a|)$ is a convex function of $x$ when $t>0$, for any $x\in [0,1]$,
\begin{align}
\psi^*(t|x-a|) & = \psi^*(t|x\cdot 1 + (1-x) \cdot 0 - a|) \\
& \leq x \psi^*(t|1-a|) + (1-x) \psi^*(t|a|).
\end{align}

Applying the inequality above, we have
\begin{align}\label{eqn.generalmarginalt}
\mathbb{E}[Z_T]  & \leq \sigma \cdot \inf_{t>0} \frac{1}{t} \left( n + \sum_{i = 1}^n P_T(i)\psi^*(t|1-a_i|) + (1-P_T(i)) \psi^*(t|a_i|) \right).
\end{align}

Now, we present the results pertaining to the $p$-norm, which corresponds to $\psi(x) = x^p, p\geq 1, x\geq 0$.

When $p = 1$, $\psi^*(y) = \begin{cases} 0 & y\in [0,1] \\ \infty & y>1 \end{cases}$. Hence, if $T$ is not deterministic, it follows from~(\ref{eqn.generalmarginalt}) that
\begin{align}
\mathbb{E}[Z_T] & \leq \sigma \inf_{a_1,a_2,\ldots,a_n \in \mathbb{R}} n \max\{|a_i|,|1-a_i|\} \\
& \leq \sigma \cdot \frac{n}{2}.
\end{align}
When $T$ is deterministic, we have $|\mathbb{E}[Z_T]| = 0$.

Now we consider the case of $p>1$. We have that
\begin{align}
\mathbb{E}[Z_T] & \leq \sigma \sum_{i = 1}^n \inf_{a_i} \norm{ P_{T|\mathbf{Z}}(i|\mathbf{Z}) - a_i  }_q,
\end{align}
since $\norm{ X}_{\psi^*}^A = \norm{X}_q$, where $\frac{1}{p} + \frac{1}{q} = 1$. Concretely,
\begin{align}
\sum_{i = 1}^n \norm{ P_{T|\mathbf{Z}}(i|\mathbf{Z}) - a_i  }_q & = \sum_{i = 1}^n \left( \mathbb{E}|P_{T|\mathbf{Z}}(i|\mathbf{Z}) - a_i|^q \right)^{1/q} \\
& = n \cdot \sum_{i = 1}^n \frac{1}{n}  \left( \mathbb{E}|P_{T|\mathbf{Z}}(i|\mathbf{Z}) - a_i|^q \right)^{1/q} \\ \
& \leq n \left( \frac{1}{n}  \sum_{i = 1}^n \mathbb{E}|P_{T|\mathbf{Z}}(i|\mathbf{Z}) - a_i|^q \right)^{1/q} \\
& = n^{1/p} \left(  \sum_{i = 1}^n \mathbb{E}|P_{T|\mathbf{Z}}(i|\mathbf{Z}) - a_i|^q \right)^{1/q},
\end{align}
where we have used the fact that $x^{1/q}, x\geq 0$ is a concave function.

It follows from Jensen's inequality that for any $x\in [0,1],q\geq 1,a\in \mathbb{R}$, we have
\begin{align}
|x-a|^q & = |x \cdot 1 + (1-x) \cdot 0 - a|^q \\
& \leq x |1-a|^q + (1-x) |a|^q,
\end{align}
and the inequality is tight when $x=1$ or $x = 0$. Applying the inequality above, we have
\begin{align}
\mathbb{E}|P_{T|\mathbf{Z}}(i|\mathbf{Z}) - a_i|^q & \leq \mathbb{E}\left[ P_{T|\mathbf{Z}}(i|\mathbf{Z}) |1-a_i|^q + (1-P_{T|\mathbf{Z}}(i|\mathbf{Z})) |a_i|^q \right] \\
& = P_T(i)|1-a_i|^q + (1-P_T(i))|a_i|^q.
\end{align}

Hence, we have that
\begin{align}
\mathbb{E}[Z_T] & \leq \sigma \cdot n^{1/p} \inf_{a_1,a_2,\ldots,a_n\in \mathbb{R}} \left( \sum_{i = 1}^n P_T(i)|1-a_i|^q + (1-P_T(i))|a_i|^q \right)^{1/q}.
\end{align}

It follows from Lemma~\ref{lemma.optimization} that
\begin{align}
\mathbb{E}[Z_T] & \leq \sigma \cdot n^{1/p} \left( \sum_{i = 1}^n \left( P_T(i)^{1/(1-q)} + (1-P_T(i))^{1/(1-q)} \right)^{1-q} \right)^{1/q}.
\end{align}
\end{proof}

\subsection{Discussions}

We now show that the upper bound is at most $\sigma \cdot \psi^{-1}(n)$. Choosing $a_i = 0, 1\leq i\leq n$, then for any $t>0$,
\begin{align}
\mathbb{E}[Z_T] & \leq \sigma \cdot \sum_{i = 1}^n  \frac{1 + \mathbb{E}[\psi^*(t|P_{T|\mathbf{Z}}(i|\mathbf{Z})|)]}{t} \\
& = \sigma \frac{1}{t} \left( n + \mathbb{E}\left[\sum_{i = 1}^n \psi^*(t|P_{T|\mathbf{Z}}(i|\mathbf{Z})|)\right] \right).
\end{align}
Since $\psi^*(x),x\geq 0$ is a convex function, and $\sum_{i = 1}^n t|P_{T|\mathbf{Z}}(i|\mathbf{Z})| \leq t$, we know that it holds pointwise that
\begin{align}
\sum_{i = 1}^n \psi^*(t|P_{T|\mathbf{Z}}(i|\mathbf{Z})|) & \leq \psi^*(t).
\end{align}
Hence, we have
\begin{align}
\mathbb{E}[Z_T] & \leq \sigma \inf_{t>0} \frac{n + \psi^*(t)}{t} \\
& = \sigma \cdot \psi^{-1}(n).
\end{align}
\section{Acknowledgement}

We are grateful to Tiancheng Yu for insightful discussions and the help in the preparation of this manuscript. 

\appendix
\section{Auxiliary lemmas} \label{sec.auxiliary}

\begin{lemma}[Generalized H\"{o}lder's Inequality]\cite{hudzik2000amemiya}
Denote an Orlicz function by $\psi$ and its convex conjugate by $\psi^* = \sup\{uv - \psi(v): v\geq 0\}$. Then,
\begin{align}
\mathbb{E}[XY] \leq \norm{ X }_\psi \norm{ Y }_{\psi^*}^A.
\end{align}
\end{lemma}

\begin{lemma}\label{lemma.optimization}
For fixed $a\in [0,1], q\geq 1, q\in \mathbb{R}$, we have
\begin{align}
\min_{x\in [0,1]} a (1-x)^q + (1-a)x^q & = \begin{cases} \left( a^{1/(1-q)} + (1-a)^{1/(1-q)} \right)^{1-q} & q > 1 \\ \min\{a, 1-a\} & q = 1 \end{cases}
\end{align}
\end{lemma}

\begin{proof}
Introduce
\begin{align}
f(x) & = a (1-x)^q + (1-a)x^q.
\end{align}
Taking derivative on both sides with respect to $x$, we have
\begin{align}
f'(x) & = -aq(1-x)^{q-1} + (1-a) q x^{q-1}.
\end{align}

From now on we only consider $q >1$, since it is clear that
\begin{align}
\lim_{q\to 1^+} \left( a^{1/(1-q)} + (1-a)^{1/(1-q)} \right)^{1-q} & = \min\{a,1-a\} \\
& = \min_{x\in [0,1]} a(1-x) + (1-a)x.
\end{align}

For any $q>1$, $f(x)$ is monotonically decreasing for any $x\leq x^*$, and then it is monotonically increasing for $x\geq x^*$. It attains the minimum when $x = x^*$, where $f'(x^*) = 0$.

Solving $f'(x^*) = 0$, we obtain that
\begin{align}
\frac{x^*}{1-x^*} & = \left( \frac{a}{1-a} \right)^{1/(q-1)},
\end{align}
which implies that
\begin{align}
f(x^*) & = \left( a^{1/(1-q)} + (1-a)^{1/(1-q)} \right)^{1-q}.
\end{align}
\end{proof}

\section{Proofs of classical maximal inequalities}\label{sec.classicalproof}

\subsection{Proof of Lemma~\ref{lemma.mgfhard}}
We have the following chain of inequalities. For any $\lambda \in [0,b)$,
\begin{align}
e^{\lambda \mathbb{E}[\max_i Z_i]} & \leq \mathbb{E}[e^{\lambda \max_i Z_i}] \\
& = \mathbb{E}[\max_i e^{\lambda Z_i}] \\
& \leq \sum_{i = 1}^n \mathbb{E}[e^{\lambda Z_i}] \\
& \leq n \cdot e^{\psi(\lambda)}.
\end{align}
Taking logarithm on both sides, we have
\begin{align}
\mathbb{E}[\max_i Z_i] & \leq \inf_{\lambda \in (0,b)} \left( \frac{\ln n + \psi(\lambda)}{\lambda} \right) \\
& = \psi^{*-1}(\ln n),
\end{align}
where in the last step we have used the fact that
\begin{align}
\psi^{*-1}(y) & = \inf_{\lambda \in (0,b)} \left( \frac{y + \psi(\lambda)}{\lambda} \right)
\end{align}
as shown in~\cite[Lemma 2.4, Pg 32]{boucheron2013concentration}.

\subsection{Proof of Lemma~\ref{lemma.orliczhard}}

We have the following chain of inequalities:
\begin{align}
\psi \left( \mathbb{E}\left[ \max_i \frac{|Z_i|}{\sigma} \right] \right) & \leq \mathbb{E}\left[ \psi\left( \max_i \frac{|Z_i|}{\sigma} \right) \right] \\
& \leq \sum_{i = 1}^n \mathbb{E}\left[ \psi \left( \frac{|Z_i|}{\sigma} \right) \right] \\
& \leq n.
\end{align}
Hence,
\begin{align}
\mathbb{E}[\max_i Z_i] & \leq \mathbb{E}[\max_i |Z_i|] \\
& \leq \sigma \cdot \psi^{-1}(n),
\end{align}
where in the last step we used the fact that an Orlicz function is nondecreasing.

\bibliographystyle{IEEEtran}
\bibliography{di}
\end{document}